\documentclass[final,1p,times]{elsarticle}
\usepackage{amssymb}
\usepackage{amsmath}
\usepackage{amsthm}
\usepackage[utf8]{inputenc}
\usepackage{marginnote}
\usepackage{amsfonts}
\usepackage{textcomp}
\usepackage{color,soul}
\usepackage{placeins}
\usepackage{tikz}
\usepackage{makecell}
\usepackage{amssymb}
\usepackage{graphicx}
\usepackage{hyperref}
\usepackage[none]{hyphenat}
\usepackage{lipsum}
\newtheorem{thm}{Theorem}[section]

\newtheorem{rem}[thm]{Remark}
\newtheorem{prop}[thm]{Proposition}
\newtheorem{cor}[thm]{Corollary}

\newdefinition{defn}[thm]{Definition}
\allowdisplaybreaks

\journal{}
\begin{document}
	\begin{frontmatter}
	
		\title{ Robustness of infinite frames and Besselian structures }
		
		\author{Shankhadeep Mondal} \ead{shankhadeep.mondal@ucf.edu}

        \author{Geetika Verma} \ead{geetika.verma@rmit.edu.au}
		
        \author{Ram Narayan Mohapatra}\ead{Ram.Mohapatra@ucf.edu}

		\address{School of Mathematics, University of Central Florida, Orlando, Florida-32816}
        \address{School of Science, Royal Melbourne Institute of Technology, Melbourne VIC 3000, Australia}
        \address{School of Mathematics, University of Central Florida, Orlando, Florida-32816}

		\begin{abstract}
			This paper extends the concepts of Minimal Redundancy Condition (MRC) and robustness of erasures for infinite frames in Hilbert spaces. We begin by establishing a comprehensive \break framework for the MRC, emphasizing its importance in ensuring the stability and resilience of frames under finite erasures. Furthermore, we discussed the robustness of erasures, which generalizes the ability of a frame to withstand information loss. The relationship between robustness, MRC, and excess of a frame is carefully examined, providing new insights into the interplay between these properties. The robustness of Besselian frames, highlighting their potential in applications where erasure resilience is critical. Our results contribute to a deeper understanding of frame theory and its role in addressing challenges posed by erasure recovery.
		\end{abstract}

		
	\end{frontmatter}
	
	\section{Introduction}
	
The use of frames in data transmission has gained prominence due to their inherent redundancy properties, which play a critical role in ensuring the integrity and precision of \break transmitted data. During the transmission process, various adverse global network conditions, such as congestion and limitations of the transmission channel, can cause errors. These errors may lead to the loss of certain parts of the encoded coefficients of a signal, resulting in what is known as erasures. The presence of redundancy within frames helps counteract these erasures and enhances the accuracy of the reconstructed signal.

    In recent years, significant research efforts have been focused on addressing the problem of erasures by applying frame theory. The concept of Minimal Redundancy Condition (MRC) plays a pivotal role in understanding and analyzing the robustness and stability of frames in signal processing and related applications. In essence, if a finite subset $\Lambda$ satisfies the MRC for a frame $\{ f_n\}_{n=1}^\infty$ in a Hilbert space $\mathcal{H},$ the remaining elements of the frame $\{ f_n\}_{n \notin \Lambda}$ continue to form a valid frame for $\mathcal{H}$. This property is central to many reconstruction and erasure recovery methods discussed in the literature. Approaches include the recovery of lost coefficients via techniques such as bridging methods \cite{larson} or erasure recovery matrices \cite{han,han1}, the inversion of partial reconstruction operators determined by subsets of frame coefficients \cite{larson}, and constructing dual frames for the reduced frame after erasures \cite{aram}. Computational efficiency in high-dimensional spaces further motivates the development of methods that focus on the erasure set, reducing the problem dimension while maintaining effective reconstruction \cite{aram,larson}.  One notable study by Casazza and Kovačević \cite{casa2} has delved into equal norm tight frames, exploring their properties, construction, and robustness to erasures in detail. Their work has provided a comprehensive understanding of how these frames can withstand erasures and maintain signal integrity. Another important contribution is from Goyal, Kovačević, and Kelner \cite{goya}, who examined uniform tight frames from the perspective of coding theory. They identified that these frames are optimal for one erasure, providing a solid foundation for further exploration in this area. Optimal dual frames are essential in minimizing reconstruction error introduced by erasures. Optimal dual frames, which minimize reconstruction error from preserved coefficients, have also been studied extensively \cite{holm,jins,jerr, sali, peh, dev, deep, shan2}. A new measure averaging error over all possible locations was introduced in \cite{bodm1}, which we adopt in this work, though our study considers all frame and dual frame pairs rather than limiting to Parseval frames and their canonical duals. Additionally, \cite{leng3} utilized weighted error operators to address probabilistic erasure scenarios, and \cite{li1} provided simplified conditions for dual frame optimality. The characterization of optimal dual frames among Parseval frames \cite{leng, li} and for general cases \cite{shan} further enriches this area. A \break classification of frames robust to a fixed number of erasures studied in\cite{casa2}. Notably, full spark frames, where any subset of indices up to a certain size satisfies the MRC, underscore the interplay between redundancy and robustness, offering powerful tools for applications requiring resilience to data loss \cite{alex,aram1}.

Building upon the concept of MRC, robustness against erasures highlights the structural resilience of frames in handling the loss of frame elements. A frame $F= \{f_i\}_{i=1}^\infty$ is said to be $m-$erasure robust if the removal of any $m$ elements leaves the remaining set as a frame for the Hilbert space. This property not only ensures stability in the reconstruction process but also establishes a direct link with the minimal redundancy condition (MRC), as both focus on maintaining sufficient information for spanning the space \cite{larson, aram2}. Moreover, the excess of a frame is an intrinsic measure of its redundancy and robustness. Frames with excess $m-$inherently satisfy $m-$erasure robustness since the removal of any $m$ elements still results in a frame \cite{balan}. Our results further this perspective by investigating how $m-$erasure robustness relates to the linear independence of associated matrices, extends to projections of frames, and applies to Besselian frames, thereby offering a comprehensive understanding of robustness in diverse infinite frame structures.

The structure of this paper is organized as follows. Section 2 provides the necessary \break preliminaries on frames and erasures, setting the groundwork for the subsequent discussions. In Section 3, we introduce the concept of the Minimal Redundancy Condition (MRC) for infinite frames, offering insights into its role in the theory of erasures. Section 4 extends the discussion to a more general framework by exploring the robustness of erasures. Here, we establish connections between the notions of robustness, MRC, and the excess of a frame, providing a unified perspective. Additionally, we delve into the robustness properties of Besselian frames, further enriching the understanding of erasure recovery in this context.
	
	\section{Preliminaries }

Let $\mathcal{H}$  be an Hilbert space.  A finite sequence $F= \{f_i\}_{i\in \mathcal{I}}$ of vectors  in $\mathcal{H}$ is called a \textit{frame} for $\mathcal{H}$ if there exist constants $A,B >0$ such that
	$$\displaystyle{A \left\| f \right\|^2\leq \sum_{i\in \mathcal{I}}\big|\langle f,f_i\rangle\big|^2 \leq B\left\| f \right\|^2}, \,\text{for all}\; f\in \mathcal{H}.$$
	
	\noindent
	The constants $A$ and $B$ are called lower and upper frame bounds respectively.  The  \textit{optimal lower frame bound} is the supremum over all lower frame bounds and the \textit{optimal upper frame bound} is the infimum over all upper frame bounds. If $A=B,$ i.e.  $\displaystyle{\sum_{i\in \mathcal{I}}\big| \langle f,f_i \rangle \big|^2 = A\left\| f \right\|^2},$ for all $f\in \mathcal{H},$  then $\{f_i\}_{i=1}^N$ is called a \textit{tight frame}.
	If $A=B=1,$ then $\{f_i\}_{i=1}^N$  is called a \textit{Parseval frame}. A frame $F$ is said to be \textit{uniform frame} if norm of each frame vector are same and if additionally $\|f_i\| = 1$ for every $i$, then $F$ is called \textit{unit norm} frame. Given a frame $F$ for a Hilbert space $\mathcal{H}$, the linear operator $\Theta_F: \mathcal{H} \to \mathbb{C}^{|\mathcal{I}|}$ defined by  
\[
\Theta_F(f) = \{\langle f, f_i \rangle\}_{i \in \mathcal{I}}
\]  
is known as the \textit{analysis operator} associated with $F$. The adjoint operator $\Theta_F^*: \mathbb{C}^{|\mathcal{I}|} \to \mathcal{H}$, given by  
\[
\Theta_F^*\bigg(\{c_i\}_{i \in \mathcal{I}}\bigg) = \sum_{i \in \mathcal{I}} c_i f_i,
\]  
is referred to as the \textit{synthesis operator} or the \textit{preframe operator} corresponding to $F$.

	The operator $S_F : \mathcal{H}\to \mathcal{H} $  defined by
	$$S_{F}f= \Theta_{F}^{*} \Theta_{F} f = \sum_{i\in \mathcal{I}} \langle f,f_i \rangle f_i $$
	is called the \textit{frame operator} associated with $F.$
	It is well known that $S_F$ is a positive, self-adjoint and invertible operator on $\mathcal{H}.$	A frame $G=\{g_i\}_{i \in \mathcal{I}}$ for $\mathcal{H}$ is called a dual frame of $F= \{f_i\}_{i \in \mathcal{I}}$ if every element $f \in  \mathcal{H}$ can be written as
	\begin{align} \label{eqn2point1reconstruction}
		f = \sum_{i \in \mathcal{I}}\langle f,f_i \rangle g_i,  \;\;   \forall f\in \mathcal{H}.
	\end{align}
	This implies that $\Theta_{G}^* \Theta_{F} = I$ and hence, $\Theta_{F}^* \Theta_{G} = I.$ So, we also have $	f = \sum\limits_{i \in \mathcal{I}}\langle f,g_i \rangle f_i,  \;\;   \forall f\in \mathcal{H}.$ In other words, $F$ is a dual of $G$ if and only if $G$ is a dual of $F.$  It is known that for any frame $F,$ the frame $\{S_{F}^{-1}f_i\}_{i \in \mathcal{I}}$ is  a dual frame of $F$ and is called the canonical or standard dual frame of $F$. In fact, it is the only dual for a frame $F,$ when $F$ is a basis.  On the other hand, if $F$ is not a basis, then  there exist infinitely many dual frames  $G$ for $F$ and every dual frame $G=\{g_i\}_{i \in \mathcal{I}}$ of F is of the form $G=\{S_{F}^{-1}f_i +u_i\}_{i \in \mathcal{I}},$ where the sequence $\{u_i\}_{i \in \mathcal{I}}$ satisfies, $$\sum_{i \in \mathcal{I}} \langle f,u_i \rangle f_i = \sum_{i \in \mathcal{I}} \langle f,f_i \rangle u_i  = 0, \;\;   \forall f\in \mathcal{H}.$$

    Given dual pair $(F,G)$ and $(F',G'),$ We define them as equivalent if there exists a unitary operator $U$ (orthogonal in the real case) such that $(F',G') = (UF,UG)$.	For a detailed study on frames, we refer to \cite{ole}. \\
	Another important operator associated with a frame is the \textit{Gramian operator}. For a frame \break $F= \{f_i\}_{i=1}^N,$ the Gramian operator $\mathcal{G}: \mathbb{C}^N \to  \mathbb{C}^N $ is defined as $\mathcal{G} = \Theta_{F} \Theta_{F}^*.$ 
    The matrix representation of the Gramian of a frame $F= \{f_i\}_{i=1}^N$ is called the \textit{Gramian matrix} and is defined by
    \begin{center}
        
  $\mathcal{G}= \begin{pmatrix}
		\langle f_1,f_1 \rangle & \langle f_2,f_1 \rangle  & \cdots  &\langle f_N,f_1 \rangle\\
		\langle f_1,f_2 \rangle & \langle f_2,f_2 \rangle  & \cdots  &\langle f_N,f_1 \rangle \\
		\vdots &\vdots  &\vdots &\vdots  \\
		\langle f_1,f_N \rangle & \langle f_2,f_N \rangle  & \cdots  &  \langle f_N,f_N \rangle \\
		
	\end{pmatrix}.$
	
	\noindent
	
\end{center}
A frame $F=\{f_i\}_{i \in \mathcal{I}}$ is called a \textit{Besselian frame} if for any sequence $\{a_i\}_{i \in \mathcal{I}},$ $\sum\limits_{i \in \mathcal{I}}a_i f_i$ converges, then $\{a_i\}_{i \in \mathcal{I}}\in \ell^2.$

\section{MRC condition on infinite frames}
The Minimal Redundancy Condition (MRC) is a critical aspect of frame theory, particularly in analyzing the robustness of erasure in infinite frames. MRC ensures that a frame possesses the least amount of redundancy necessary to maintain its reconstruction properties. This condition is particularly valuable in applications requiring resilience to data loss, as it allows for optimal balance between robustness and efficiency. In the context of infinite frames, MRC plays a significant role in characterizing how subsets of frame elements contribute to stability under erasure and whether the frame maintains its spanning properties. In the infinite frame setting, MRC involves analyzing the properties of localized finite sections and their contributions to global robustness. For instance, given an infinite frame, if a finite subset of indices $\Lambda$ is removed, the remaining sequence still satisfies a modified frame inequality with adjusted bounds. \break Investigating this behavior under the MRC sheds light on the stability of the frame operator $S_F=\sum\limits_{i \in \mathcal{I}} \langle f,f_i \rangle f_i$ and its ability to retain invertibility when restricted to subspaces. Studying the robustness of erasure under the MRC offers valuable insights into designing frames with optimal efficiency, particularly in signal processing, where erasures or losses are inevitable. It also provides a pathway for understanding the interplay between frame redundancy and stability, ensuring a balance between these properties in infinite-dimensional spaces. Investigating the robustness of erasure in the MRC setting provides a deeper understanding of how frames behave under perturbations, offering insights into designing systems that are both efficient and resilient in various spaces.
  
Let $F= \{f_i\}_{i=1}^\infty$ be a frame for $\mathcal{H}.$ $F$ satisfy the \textit{minimal redundancy condition (MRC in short)} for a set $\Lambda \subset \{1,2, \ldots \infty\}$ if $\overline{\text{span}}\{f_i : i \in \Lambda^c\} = \mathcal{H}.$ In other words $\{f_i : i \in \Lambda^c\}$ is a frame for $\mathcal{H}.$ For a dual frame pair \((F, G)\), if the erasure set \(\Lambda\) satisfies the minimal redundancy condition, then \(\{g_j : j \notin \Lambda\}\) forms a frame for \(\mathcal{H}\). Consequently, there exists at least one dual frame (possibly many) that can reconstruct any \(f \in \mathcal{H}\) from the coefficients indexed by \(j \notin \Lambda\). This guarantees that sufficient information is retained in \(\{\langle f, g_j \rangle: j \notin \Lambda\}\) for accurate reconstruction. However, if the erasure set \(\Lambda\) fails to satisfy the minimal redundancy condition, there exists a nonzero vector \(f \in \mathcal{H}\) that is orthogonal to all \(g_j\) for \(j \notin \Lambda\). In such cases, it becomes impossible to reconstruct \(f\) using only the coefficients indexed by \(j \notin \Lambda\). In other words, the partial reconstruction operator $R_\Lambda(f) = \sum\limits_{i \in {\Lambda^c}} \langle f, g_i \rangle f_i $ need not be invertible. This underscores the "minimal" nature of the redundancy condition, as it delineates the boundary between reconstructability and information loss.

\noindent The following theorem ensures that if a subset of frame elements indexed by $\Lambda^{c}$ is sufficient to reconstruct elements of $\mathcal{H}$ using $F$, the same holds for the canonical dual frame $S_{F}^{-1}F.$ This symmetry is vital because the canonical dual frame plays a central role in the reconstruction of signals, particularly through the dual reconstruction formula.

\begin{thm}\label{thm3point1}
Let $F= \{f_i\}_{i=1}^\infty$ be a frame for $\mathcal{H}.$  Any finite set $\Lambda \subset \{1,2,\ldots \infty\}$ satisfy the MRC for $F$ if and only if $\Lambda$ satisfy the MRC for the canonical dual $S_{F}^{-1}F.$ 
\end{thm}

\begin{proof}
Without loss of generality let us assume $\Lambda = \{1,2,\ldots,k\}.$ Let $\Theta_{S_{F}^{-1}F}$ be the analysis operator corresponding to $S_{F}^{-1}F.$ $\Lambda$ satisfy the MRC for $F$ if and only if $\{f_i\}_{i\in \Lambda^c}$ is a frame for $\mathcal{H}$ if and only if ${\Theta_F}_{\Lambda^c}^* {\Theta_F}_{\Lambda^c}$ is invertible, where $F_{\Lambda^c} = \{f_i\}_{i \in \Lambda^c}$. Now,
\begin{align*}
    {\Theta_{S_{F}^{-1}F}}_{\Lambda^c}^*  {\Theta_{S_{F}^{-1}F}}_{\Lambda^c} f&= \sum\limits_{i \in \Lambda^c} \langle f,  S_{F}^{-1}f_i \rangle S_{F}^{-1}f_i \\&= S_{F}^{-1}\left( \sum\limits_{i \in \Lambda^c} \langle  S_{F}^{-1}f, f_i \rangle f_i \right) \\&= S_{F}^{-1} {\Theta_F}_{\Lambda^c}^* {\Theta_F}_{\Lambda^c} S_{F}^{-1}f,
\end{align*}
for all $f \in \mathcal{H}.$ Thus, $ {\Theta_{S_{F}^{-1}F}}_{\Lambda^c}^*  {\Theta_{S_{F}^{-1}F}}_{\Lambda^c} =  S_{F}^{-1} {\Theta_F}_{\Lambda^c}^* {\Theta_F}_{\Lambda^c} S_{F}^{-1}.$ Therefore, ${\Theta_{S_{F}^{-1}F}}_{\Lambda^c}^*  {\Theta_{S_{F}^{-1}F}}_{\Lambda^c} $ is invertible if and only if  $S_{F}^{-1} {\Theta_F}_{\Lambda^c}^* {\Theta_F}_{\Lambda^c} S_{F}^{-1}$ is invertible if and only if ${\Theta_F}_{\Lambda^c}^* {\Theta_F}_{\Lambda^c}$ is invertible. Hence, $\Lambda$ satisfy MRC for $S_{F}^{-1}F$ if and only if $\Lambda$ satisfy MRC for $F.$

\end{proof}
\noindent For any unitary operator $U$ on $\mathcal{H}$, it is straightforward to see that  ${\Theta_{UF}}_{\Lambda^c}^*  {\Theta_{UF}}_{\Lambda^c} = U {\Theta_F}_{\Lambda^c}^* {\Theta_F}_{\Lambda^c} U^*.$
Thus, using a similar argument as in \ref{thm3point1}, we conclude the following theorem.

\begin{thm}\label{thm3point2}
    Let $F= \{f_i\}_{i=1}^\infty$ be a frame for $\mathcal{H}$, and let $U$ be a unitary operator on $\mathcal{H}$. A finite subset $\Lambda \subset \{1,2,\ldots, \infty\}$ satisfies the MRC for $F$ if and only if $\Lambda$ satisfies the MRC for $UF$.
\end{thm}

\noindent The following proposition extends the understanding of the Minimal Redundancy Condition (MRC) from a frame \( F = \{f_i\}_{i=1}^\infty \) to its duals \( G = \{S_{F}^{-1}f_i + h_i\}_{i=1}^\infty \). It shows that if a finite set of indices \( \Lambda \) satisfies the MRC for \( F \), it also satisfies the MRC for any dual \( G \), provided a specific operator condition involving \( \Theta_F \) and the perturbation \( h \) is met.

\begin{prop}
    Let $F= \{f_i\}_{i=1}^\infty$ be a frame for $\mathcal{H}.$  Suppose  $\Lambda = \{n_1,n_2,\ldots,n_k\}$ be a finite set of indices, satisfy the MRC for $F.$ Then $\Lambda$ satisfy the MRC for a dual $G=\left\{S_{F}^{-1}f_i + h_i    \right\}_{i=1}^\infty$ of $F$ if and only if $$\left( (I - (\Theta_{F}^* \Theta_{F})^{-1}{\Theta_{F}^*}_\Lambda {\Theta_{F}}_\Lambda )(\Theta_{F}^* \Theta_{F})^{-1} - {\Theta_{h}}_\Lambda^* {\Theta_{F}}_\Lambda(\Theta_{F}^* \Theta_{F})^{-1} - (\Theta_{F}^* \Theta_{F})^{-1}{\Theta_{F}^*}_\Lambda {\Theta_{h}}_\Lambda + {\Theta_{h}}_{\Lambda^c}^*{\Theta_{h}}_{\Lambda^c}  \right)$$ is invertible, where $h= \{h_i\}_{i=1}^\infty, F_\Lambda = \{f_i\}_{i \in \Lambda}, F_{\Lambda^c} = \{f_i\}_{i \in \Lambda^c}, h_\Lambda = \{h_i\}_{i \in \Lambda} \text{and}\; h_{\Lambda^c} = \{h_i\}_{i \in \Lambda^c}.$
\end{prop}

\begin{proof}
    Let $G= \{S_{F}^{-1}f_i +h_i\}_{i=1}^\infty$ be a dual of $F,$
wherein $\sum\limits_{i=1}^\infty \langle f, f_i\rangle h_i = 0.$ Now,
\begin{align}\label{eqn3point2}
    {\Theta_G}_{\Lambda^c}^* {\Theta_G}_{\Lambda^c} f &= \sum\limits_{i \in \Lambda^c} \left\langle f, S_{F}^{-1}f_i + h_i\right\rangle \left( S_{F}^{-1}f_i + h_i\right) \nonumber\\&= \sum\limits_{i \in \Lambda^c} \left\langle f, S_{F}^{-1}f_i \right\rangle S_{F}^{-1}f_i + \sum\limits_{i \in \Lambda^c} \left\langle f, S_{F}^{-1}f_i \right\rangle h_i + \sum\limits_{i \in \Lambda^c} \langle f, h_i \rangle S_{F}^{-1}f_i + \sum\limits_{i \in \Lambda^c} \langle f, h_i \rangle h_i \nonumber\\&= S_{F}^{-1}\left( \sum\limits_{i \in \Lambda^c} \left\langle f, S_{F}^{-1}f_i \right\rangle f_i \right) + \sum\limits_{i \in \Lambda^c} \left\langle f, S_{F}^{-1}f_i \right\rangle h_i +  S_{F}^{-1}\left(  \sum\limits_{i \in \Lambda^c} \left\langle f, h_i \right\rangle f_i \right) + \sum\limits_{i \in \Lambda^c} \langle f, h_i \rangle h_i \nonumber\\&= S_{F}^{-1} {\Theta_F}_{\Lambda^c}^* {\Theta_F}_{\Lambda^c} S_{F}^{-1} f + {\Theta_h}_{\Lambda^c}^*{\Theta_F}_{\Lambda^c}S_{F}^{-1} f + S_{F}^{-1} {\Theta_F}_{\Lambda^c}^* {\Theta_h}_{\Lambda^c} f + {\Theta_h}_{\Lambda^c}^* {\Theta_h}_{\Lambda^c}f \nonumber\\&=\bigg(\left(I - (\Theta_{F}^* \Theta_F)^{-1} {\Theta_F}_{\Lambda}^* {\Theta_F}_{\Lambda}\right)(\Theta_{F}^* \Theta_F)^{-1} -  {\Theta_h}_{\Lambda}^* {\Theta_F}_{\Lambda}^*   (\Theta_{F}^* \Theta_F)^{-1} - \nonumber \\& \quad \;\;\;\;\;(\Theta_{F}^* \Theta_F)^{-1} {\Theta_F}_{\Lambda}^* {\Theta_h}_{\Lambda} + {\Theta_h}_{\Lambda^c}^* {\Theta_h}_{\Lambda^c}\bigg)f 
\end{align}
Thus, ${\Theta_G}_{\Lambda^c}^* {\Theta_G}_{\Lambda^c}$ is invertible if and only if  $$ \left(\left(I - (\Theta_{F}^* \Theta_F)^{-1} {\Theta_F}_{\Lambda}^* {\Theta_F}_{\Lambda}\right)(\Theta_{F}^* \Theta_F)^{-1} -  {\Theta_h}_{\Lambda}^* {\Theta_F}_{\Lambda}^*   (\Theta_{F}^* \Theta_F)^{-1} - (\Theta_{F}^* \Theta_F)^{-1} {\Theta_F}_{\Lambda}^* {\Theta_h}_{\Lambda} + {\Theta_h}_{\Lambda^c}^* {\Theta_h}_{\Lambda^c}\right)$$ is invertible and hence the result follows.

\end{proof}

\begin{cor}
    $\Lambda$ satisfy the MRC for $S_{F}^{-1}F$ if and only if $\left( I -  {\Theta_{S_{F}^{-1}F}}_{\Lambda}^* {\Theta_F}_{\Lambda}\right)$ is invertible.
\end{cor}
\begin{proof}
The result follows directly from equation \eqref{eqn3point2} by setting \( h_i = 0 \) for all \( i \). In this case, the additional terms vanish, reducing the condition to the case for the canonical dual frame \( S_F^{-1}F \), thus proving the Corollary.
\end{proof}

\section{Robustness of Erasures}
A frame $\{f_i\}_{i=1}^\infty$ is referred to as \emph{exact} if it becomes incomplete whenever any single element is removed. A frame \(\{\varphi_i\}_{i=1}^M\) is said to be robust to \(k\) erasures if it retains the frame property even after removing \(k\) elements from the frame. More explicitly, for any index set \(I \subset \{1, 2, \ldots, M\}\) with \(|I| = k\), the remaining set of vectors \(\{\varphi_i\}_{i \in I^c}\), where \(I^c\) is the complement of \(I\) in \(\{1, 2, \ldots, M\}\), still forms a frame for the Hilbert space \(\mathcal{H}\). This property implies that the frame is sufficiently redundant to withstand the loss of any \(k\) frame elements without compromising its ability to provide stable and unique representations of vectors in \(\mathcal{H}\). The robustness to erasures is a critical feature in applications such as signal processing, data transmission, and compressed sensing, where loss of data or components is inevitable, and maintaining the completeness of the representation is essential. In what follows, we use $\mathcal{N}(T)$ to denote the null space of an operator or matrix $T$.


Let $F = \{f_i\}_{i=1}^\infty$ be  a $m-$erasure robust  frame for $\mathcal{H}.$ Then $ \{f_i\}_{i=m+1}^\infty$ is an exact frame or basis for $\mathcal{H}.$ Consequently, for $1 \leq i \leq m,$ we can find a unique sequence $\{d_{i,j}\}_{m+1 \leq j <\infty}$ such that
\begin{equation}\label{eqn5point4}
    f_i = \sum\limits_{j=m+1}^\infty d_{i,j}\, f_j.
\end{equation}

Let us define the matrix
	$$\Gamma_{F} :=	\begin{pmatrix}
		1 & 0  & \cdots &0 & -d^{*}_{1,m+1} & \cdots  \\
        0 & 1  & \cdots &0 & -d^{*}_{2,m+1} & \cdots  \\
		\vdots &\vdots  &\vdots &\vdots &\vdots&\vdots\\
		0 & 0  & \cdots &1 & -d^{*}_{m,m+1} & \cdots \\
	\end{pmatrix}.$$
\\
\begin{thm}
    For a $m-$erasure robust  frame $F$ for $\mathcal{H},\;T\mathcal{H} = \mathcal{N}\left( \Gamma_{F}\right)$ 
\end{thm}
\begin{proof}
    Let  $\alpha = \left(\alpha_1, \alpha_2, \ldots  \right)^t \in T\mathcal{H}.$ then there exists a $f \in \mathcal{H}$ such that $Tf = \alpha.$ This implies $\langle f, f_i \rangle = \alpha_i, $ for all $i.$ As $F$ is a $m-$erasure robust frame for $\mathcal{H},$ \\
 which follows that $\text{span}\left\{f_i : m+1 \leq i < \infty \right\} = \mathcal{H}$ and hence by\eqref{eqn5point4} for $1 \leq i \leq m,$
$$\alpha_i = \langle f, f_i \rangle = \sum\limits_{j= m+1}^\infty d_{i,j}^{*} \alpha_j .$$
This gives $\Gamma_{F} (\alpha) = 0 $ and hence, $\alpha \in \mathcal{N}\left( \Gamma_{F}\right).$ \\

Conversely, $\alpha \in \mathcal{N}\left( \Gamma_{F}\right)$ implies that $\alpha_i = \sum\limits_{j= m+1}^\infty d_{i,j}^{*} \alpha_j , \, 1\leq i \leq m.$ As $ \{f_i\}_{i=m+1}^\infty$ is a basis for $\mathcal{H},$ then there exist a $f \in \mathcal{H}$ such that $\langle f, f_i \rangle= \alpha_i, \, m+1 \leq i < \infty.$ For $1 \leq i \leq m,\, \alpha_i =  \sum\limits_{j= m+1}^\infty d_{i,j}^{*} \alpha_j , \, 1\leq i \leq m.$ As $ \{f_i\}_{i=m+1}^\infty$ is a basis for $\mathcal{H},$ then there exist a $f \in \mathcal{H}$ such that $\langle f, f_i \rangle= \alpha_i, \, m+1 \leq i < \infty.$  For $1 \leq i \leq m,\,$ 
$$ \alpha_i = \sum\limits_{j= m+1}^\infty d_{i,j}^{*} \alpha_j=  \sum\limits_{j= m+1}^\infty d_{i,j}^{*} \langle f, f_j \rangle = \left\langle f, \sum\limits_{j= m+1}^\infty d_{i,j} f_j \right\rangle = \langle f, f_i \rangle. $$ 
Hence, $\alpha \in T\mathcal{H}.$

\end{proof}

The following proposition establishes a crucial connection between the robustness of a frame to \(m\)-erasures and the structural properties of an associated matrix \(\Gamma\). Specifically it states that the \(m\)-erasure robustness of the frame ensures that every subset of \(m\) columns of \(\Gamma\) is linearly independent. This underscores the frame's capability to maintain sufficient information for reconstruction even in the presence of \(m-\) erasures.

\begin{prop}
    Let $F= \{f_i\}_{i=1}^\infty$ be a $m-$erasure robust frame for $\mathcal{H}$ and $\Gamma$ be a matrix such that $T\mathcal{H} = \mathcal{N}\left( \Gamma\right).$ Then every $m$ columns of\, $\Gamma$ are linearly independent.
\end{prop}

\begin{proof}
    Without loosing the generality we assume that the erasure set is $\Lambda= \{1,2,\ldots,m  \}.$ The other cases can be proved similarly. We will prove this by contradiction. Suppose that there are some constant $c_1,c_2,\ldots,c_m$ not all zero such that $c:=\left( c_1,c_2,\ldots,c_m,\ldots\right)^t$ satisfy $\Gamma c =0.$ Using the fact $T\mathcal{H} = \mathcal{N}\left( \Gamma\right),$ we can say that there exist a $f \in \mathcal{H}$ such that 
    \begin{equation}\label{eqn5point5}
        \langle f,f_i\rangle = \begin{cases}
			c_i,\;\;   & 1\leq i \leq m,\\
			0,\;\;\; & m+1 \leq i <\infty.
		\end{cases}
	 \end{equation}
    	
        Since $\{f_i\}_{i= m+1}^\infty$ is an exact frame for $\mathcal{H},$ therefore it spans $\mathcal{H}$ and hence $f=0.$ This leads to $c_i =0,\,$ for all $1 \leq i \leq m,$ which is not possible. Thus the first $m$ columns of $\Gamma $ are linearly independent. 
        \end{proof}

\noindent 
The following proposition highlights the preservation of robustness under orthogonal projections. Specifically, it shows that if a frame is $m$-erasure robust for a Hilbert space $\mathcal{H}$, then its projection onto a subspace $P\mathcal{H}$ retains the same $m$-erasure robustness. This result underscores the stability of frame robustness under linear transformations like orthogonal projections, making it a valuable property in applications such as signal processing and dimensionality reduction.

\begin{prop}
    Let $F= \{f_i\}_{i=1}^\infty$ be a  frame for $\mathcal{H}$ which is $m-$erasure robust and $P$ be an Orthogonal projection on $\mathcal{H}.$ Then, $PF$ is a $m-$erasure robust frame for $P\mathcal{H}.$
\end{prop}
\begin{proof}
    For any $\Lambda= \{\ell_1, \ell_2,\ldots,\ell_m  \}\subset \{1,2,\ldots \infty  \},\,$ $F_\Lambda = \{f_i\}_{i \in \Lambda^c}$ is a frame for $\mathcal{H}.$ Therefore, $F_\Lambda$ has a dual $G_\Lambda = \{g_i\}_{i \in \Lambda^c}$ such that  
    $$f = \sum\limits_{i\in \Lambda^c}\langle f, f_i \rangle g_i = \sum\limits_{i\in \Lambda^c}\langle f, g_i \rangle f_i,$$
    for all $f \in \mathcal{H}.$ This leads to $Pf = \sum\limits_{i\in \Lambda^c}\langle Pf, f_i \rangle g_i = \sum\limits_{i\in \Lambda^c}\langle f, Pf_i \rangle g_i,$ using the first equality and $Pf = \sum\limits_{i\in \Lambda^c}\langle f, g_i \rangle P f_i,$ using the second equality and for all $f \in \mathcal{H}.$ Therefore, $\{Pf_i\}_{i \in \Lambda^c}$ is a frame for $P\mathcal{H}.$
\end{proof}

A theorem analogous to Theorem \ref{thm3point2} states as follows: 

\begin{thm}
 Let $U$ be a unitary operator on $\mathcal{H}$. A frame  $F= \{f_i\}_{i=1}^\infty$ on $\mathcal{H}$ is a $m-$erasure robust if and only if $UF= \{Uf_i\}_{i=1}^\infty$  is also $m-$erasure robust.
\end{thm}

A frame $F = \{f_i\}_{i=1}^\infty$ for a Hilbert space $\mathcal{H}$ is called a frame with excess $m$ if for any subset $\Lambda \subset \{1, 2, \ldots, \infty\}$ with $\lvert \Lambda \rvert = m$, it holds that 
\[
f_i = \sum_{j \notin \Lambda} a_{ij} f_j, \quad \text{for all } i \in \Lambda,
\]
where $a_{ij}$ are suitable coefficients. In other words, The excess of a frame $F$ is defined as 
\[
\mathcal{E}(F) = \sup\left\{\lvert \alpha \rvert : \alpha \subset F \text{ and } \overline{\text{span}}(F \setminus \alpha) = \overline{\text{span}}(F)\right\}.
\]

\noindent The concept of excess in a frame is closely connected to both the MRC and robustness. A frame with excess $m$ ensures that any subset of $m$ elements can be removed while the remaining elements still span the Hilbert space, satisfying the MRC for subsets of size $m.$ This redundancy guarantees that the frame retains its structure and completeness despite the removal of specific elements. Consequently, a frame with excess $m$ is also inherently robust to $m-$erasures, meaning that the removal of up to $m$ elements does not prevent the remaining elements from forming a valid frame. This dual connection highlights how excess encapsulates both redundancy and stability, ensuring resilience to data loss and enabling robust signal reconstruction in practical applications.

\begin{prop}

    Let $F = \{f_i\}_{i=1}^\infty$ be a frame for a Hilbert space $\mathcal{H}$. Then $F$ has excess $m$ if and only if for any subset $\Lambda \subset \mathbb{N}$ with $\lvert \Lambda \rvert = m$, the set $F_c = \{f_i\}_{i \in \Lambda^c}$ is also a frame for $\mathcal{H}$.
\end{prop}

\begin{proof}
    Let $F_c = \{f_i\}_{i \in \Lambda^c}$ is a frame for $\mathcal{H}$ for any $\Lambda \subset \{1, 2, \ldots, \infty\}$ with $\lvert \Lambda \rvert = m.$ Thus $\text{span} \left\{ f_i  \right\}_{i \in \Lambda^c} = \mathcal{H}.$ Therefore, for $i \in \Lambda,\;\,$ there exists a sequence $\left\{a_{ij}\right\}_{j \in \Lambda^c}$ such that $f_i = \sum\limits_{j \in \Lambda^c} a_{ij} f_j, \quad \text{for all } i \in \Lambda.$ Conversely, Suppose  $F = \{f_i\}_{i=1}^\infty$ be a frame with excess $m.$ Then,  for all $i \in \Lambda,\;\,f_i = \sum\limits_{j \notin \Lambda} a_{ij} f_j.$ Now, 
    \begin{align*}
        \displaystyle\sum\limits_{i =1}^\infty \big|\langle f,f_i\rangle\big|^2 &= \sum\limits_{i \in \Lambda}\big|\langle f,f_i\rangle\big|^2 + \sum\limits_{i \in \Lambda^c}\big|\langle f,f_i\rangle\big|^2 \\&= \displaystyle\sum\limits_{i \in \Lambda}\bigg|\left\langle f, \sum\limits_{j \in \Lambda^c} a_{ij} f_j\right\rangle\bigg|^2 + \sum\limits_{i \in \Lambda^c}\big|\langle f,f_i\rangle\big|^2 \\&= \sum\limits_{i \in \Lambda}\sum\limits_{j \in \Lambda^c}\left|a_{ij}\right|^2 \big|\langle f,f_j\rangle\big|^2 + \sum\limits_{i \in \Lambda^c}\big|\langle f,f_i\rangle\big|^2 \\&\leq a \sum\limits_{i \in \Lambda}\sum\limits_{j \in \Lambda^c} \big|\langle f,f_j\rangle\big|^2 + \sum\limits_{i \in \Lambda^c}\big|\langle f,f_i\rangle\big|^2 \\&= (am+1) \sum\limits_{i \in \Lambda^c}\big|\langle f,f_i\rangle\big|^2,
    \end{align*}
    where $a= \max \left|a_{ij}\right|^2.$ Thus, $\displaystyle\sum\limits_{i \in \Lambda^c}\big|\langle f,f_i\rangle\big|^2 \geq \dfrac{1}{(am+1)}  \sum\limits_{i =1}^\infty \big|\langle f,f_i\rangle\big|^2 \geq  \dfrac{A}{(am+1)}\|f\|^2$ and hence, $\left\{  f_i\right\}_{i \in \Lambda^c}$ is a frame for $\mathcal{H}.$
\end{proof}

\begin{cor}

    Let $F = \{f_i\}_{i=1}^\infty$ be a frame for a Hilbert space $\mathcal{H}$ with excess $m$. Then $F$ satisfies the $m$-erasure minimal redundancy condition (MRC) for any subset $\Lambda \subset \{1, 2, \ldots, \infty\}$ with $\lvert \Lambda \rvert = m$ and thus $F$ is a $m-$erasure robust frame for $\mathcal{H}.$

\end{cor}

\subsection{\textbf{Robustness of Besselian Frame}}

A frame $F = \{f_i\}_{i = 1}^\infty$ in a Hilbert space $\mathcal{H}$ is called a \textit{Besselian frame} if, whenever the series $\sum\limits_{i=1}^\infty a_i f_i$ converges, it follows that $\{a_i\}_{i=1}^\infty \in \ell^2$ \cite{holub}. Let $\Lambda$ be the erased set for the Besselian frame $F=\{f_i\}_{i=1}^\infty$. Let $f \in \mathcal{H}.$ Now we define the bridging set for $\Lambda.$ We replace each erased coefficient $\langle f,g_i\rangle ,$ for $j \in \Lambda$ with$\langle f,g'_i\rangle ,$ for $g'_j \in \text{span}\left\{ g_k: k \in \delta\right\}$ for some $\delta \subset \Lambda^c.$ Any set $\delta \subset \Lambda^c$ is called the bridge set. The partial reconstruction with bridging set is $\tilde{f} = f_{R} + f_{B},$ where $f_B = \sum\limits_{j \in \Lambda}\langle f, g'_j \rangle f_j$ and $f_B$ is called the \textit{bridging supplement} and $B_\Lambda:= \langle .,g'_j \rangle f_j$ is called the \textit{bridging supplement operator}. Then the reduced error is $f_{\tilde{E}} := f- \tilde{f}$ and the associated reduced error operator is $\tilde{E_{\Lambda}}= I - R_\Lambda - B_\Lambda.$ So, $\tilde{E_{\Lambda}}f = \sum\limits_{j \in \Lambda}\langle f, g_j - g'_j \rangle f_j.$ Let $g'_k = \sum\limits_{\ell \in \delta} c_{\ell}^{(k)}g_{\ell}.$ So, we have 
$$\langle f_j, g_k \rangle- \sum\limits_{\ell \in \delta} \overline{c_{\ell}^{(k)}}\langle f_j, g_{\ell} \rangle = 0.$$
For all $k\in \Lambda,$ we then have $\langle f_j, g_k \rangle =\sum\limits_{\ell \in \delta} c_{\ell}^{(k)}g_{\ell}.$ Let $\Lambda = \left\{ \lambda_j \right\}_{j=1}^L$ and $\delta = \left\{ \delta_j \right\}_{j=1}^M.$ We have the matrix equation 
\begin{equation}\label{eqn4point3}
    	\begin{pmatrix}
		\langle f_{\lambda_1}, g_{\delta_1} \rangle & \langle f_{\lambda_1}, g_{\delta_2} \rangle & \cdots &\langle f_{\lambda_1}, g_{\delta_M} \rangle   \\
        \langle f_{\lambda_2}, g_{\delta_1} \rangle & \langle f_{\lambda_2}, g_{\delta_2} \rangle & \cdots &\langle f_{\lambda_2}, g_{\delta_M} \rangle   \\
		\vdots &\vdots  &\vdots &\vdots\\
		\langle f_{\lambda_L}, g_{\delta_1} \rangle & \langle f_{\lambda_L}, g_{\delta_2} \rangle & \cdots &\langle f_{\lambda_L}, g_{\delta_M} \rangle  \\
	\end{pmatrix} 
    \begin{pmatrix}
	    \overline{c_{\delta_1}^{(k)}} \\ \overline{c_{\delta_2}^{(k)}} \\ \vdots \\ \overline{c_{\delta_M}^{(k)}} 
	\end{pmatrix} = \begin{pmatrix}
	    \langle f_{\lambda_1}, g_{k} \rangle \\ \langle f_{\lambda_2}, g_{k} \rangle \\ \vdots \\\langle f_{\lambda_L}, g_{k} \rangle
	\end{pmatrix}
\end{equation}
    The matrix $\begin{pmatrix}
		\langle f_{\lambda_1}, g_{\delta_1} \rangle & \langle f_{\lambda_1}, g_{\delta_2} \rangle & \cdots &\langle f_{\lambda_1}, g_{\delta_M} \rangle   \\
        \langle f_{\lambda_2}, g_{\delta_1} \rangle & \langle f_{\lambda_2}, g_{\delta_2} \rangle & \cdots &\langle f_{\lambda_2}, g_{\delta_M} \rangle   \\
		\vdots &\vdots  &\vdots &\vdots\\
		\langle f_{\lambda_L}, g_{\delta_1} \rangle & \langle f_{\lambda_L}, g_{\delta_2} \rangle & \cdots &\langle f_{\lambda_L}, g_{\delta_M} \rangle  \\
	\end{pmatrix} $ is called the \textit{bridge matrix}. \\
    
   \noindent We notate this matrix by  $\Xi(F,G,\Lambda, \delta).$ Therefore the equation \eqref{eqn4point3} can be written as\\ $\Xi(F,G,\Lambda, \delta) C = \Xi(F,G,\Lambda, \Lambda).$ For further details on the bridging set, please refer to \cite{larson}. 
    \begin{rem}\cite{larson}
       For a dual frame pair $(F, G)$ for $\mathcal{H}$, there exists a bridge set $\delta$ for an erased set $\Lambda$ if and only if $\Lambda$ satisfies the minimal redundancy condition.
   \end{rem}

\noindent The following proposition establishes a fundamental equivalence for Besselian frames in the context of robustness to single erasures. Specifically, it reveals that a Besselian frame $F$ is robust to a single erasure if and only if there exists a sequence of non-zero scalars ${a_i}$ belonging to $\ell^2$ such that their weighted sum with the frame elements results in the zero vector. This equivalence provides a concrete characterization of $1$-erasure robustness, connecting it to the existence of such a sequence, which inherently reflects a balance within the frame's structure. The result bridges the concepts of robustness and the frame's internal linear dependencies, offering a deeper insight into how Besselian frames sustain their functional properties despite single-point erasures.

\begin{thm} \label{prop3point1be1robust}
    Let $F=\{f_i\}_{i=1}^\infty$ be a Besselian frame for $\mathcal{H}.$ Then the following are equivalent:
    \begin{enumerate}
			\item [{\em (i)}] $F$ be a Besselian frame robust to $1-$erasure.
			\item [{\em (ii)}] There exists a sequence of non-zero scalers $\{a_i\}$ such that $\{a_i\}_{i=1}^\infty \in \ell^2$  and $\sum\limits_{i=1}^\infty a_i f_i = 0.$ 
		\end{enumerate}
\end{thm}

\begin{proof}
    $(i) \implies (ii):$\;
    Let $\mathcal{J} \subset \left\{ 1,2,\ldots \infty  \right\}$ such that $\mathcal{J}$ is maximal for which there are non-zero $a_i$'s, $i \in \mathcal{J}$ and 
    \begin{equation}\label{eqn2}
        g = \sum \limits_{i \in \mathcal{J} } a_i f_i = 0.
    \end{equation}
    We claim that $\mathcal{J} = \left\{ 1,2,\ldots \infty  \right\}.$    Suppose not. Let $m \in \mathcal{J}^c, $ where $\mathcal{J}^c$ is the complement of $\mathcal{J}.$ By our assumption, $\{f_i\}_{i=1}^\infty$ is a robust for $1-$erasure. Then, there exists $\{b_i\}_{i \neq m}$ not all zero, such that 
    $$ f_m = \sum \limits_{i \neq m} b_i f_i $$

    So,
    \begin{equation} \label{eqn3}
     0= h=  f_m - \sum \limits_{i \neq m} b_i f_i = f_m - \sum \limits_{i \in \mathcal{J}} b_i f_i - \sum \limits_{m \neq i \in \mathcal{J}^c} b_i f_i
     \end{equation}

     Case 1: If $b_i = 0, \forall i \in \mathcal{J}.$     Then, $0= h= f_m - \sum \limits_{m \neq i \in \mathcal{J}^c} b_i f_i.$   This leads to,
     $$ 0 = g+ h = \sum\limits_{i \in \mathcal{J}}a_i f_i + f_m - \sum \limits_{m \neq i \in \mathcal{J}^c} b_i f_i.$$
     Hence, $m \in \mathcal{J},$ which gives a contradiction. Therefore, $\mathcal{J} = \left\{ 1,2,\ldots \infty  \right\}.$\\

     Case 2: At least one $b_i \neq 0,$ for some $i \in \mathcal{J}.$ Since $a_i \neq 0, \forall i \in \mathcal{J},$ we can choose $\epsilon > 0$ such thst $\epsilon \neq \frac{a_i}{b_i},\,\forall i \in \mathcal{J}.$ Note that, existence of such $\epsilon$ is guaranteed by the fact that the set of all positive real number is uncountable.    Also, $g + \epsilon h = 0.$ This can be re-written as
     $$ \sum\limits_{i \in \mathcal{J}} a_i f_i + \epsilon \left( f_m - \sum\limits_{i \in \mathcal{J}} b_i f_i - \sum\limits_{m \neq i \in \mathcal{J}^c} b_i f_i \right) = 0. $$

     This leads to,
     $$ \sum\limits_{i \in \mathcal{J}} \left( a_i - \epsilon b_i   \right) f_i  + \epsilon f_m - \epsilon \sum\limits_{m \neq i \in \mathcal{J}^c} b_i f_i  = 0$$

     Hence, $m \in \mathcal{J},$ which is not possible. Therefore, $\mathcal{J} = \left\{ 1,2,\ldots \infty  \right\}.$\\

     $(ii) \implies (i):$
     This side is trivial, as, for any $1 \leq m < \infty,$ $f_m = \sum\limits_{i \neq m} - \left( \frac{a_i}{a_m} \right)f_i.$
     
\end{proof}

The next proposition provides a comprehensive characterization of $1$-erasure robustness for Besselian frames derived from an orthonormal basis under a projection. It establishes that robustness is equivalent to three conditions: the existence of a non-zero sequence $\{a_i\}_{i=1}^\infty$ whose weighted sum of projected basis elements vanishes, or the presence of a vector in the orthogonal complement of the projection's range that maintains non-zero inner products with all basis elements. This equivalence highlights the intricate balance between projection, frame structure, and robustness, offering both geometric and analytical perspectives on the preservation of frame properties under erasures.

\begin{prop}
    Let $\{e_i\}_{i=1}^\infty$ be an Orthonormal basis for $\mathcal {H}.$ Then the following are equivalent:
     \begin{enumerate}
			\item [{\em (i)}] $\{Pe_i\}_{i=1}^\infty$ be a Besselian frame robust to $1-$erasure.
			\item [{\em (ii)}] There exists a sequence of non-zero scaler $\{a_i\} \in \ell^2$ such that  $\sum\limits_{i=1}^\infty a_i P e_i = 0.$ 
   \item [{\em (iii)}] There is a $f \in (I-P)\mathcal{H}$ such that $\langle f, e_i \rangle \neq 0, \forall i.$
		\end{enumerate}
    
\end{prop}

\begin{proof}
    $(i) \iff (ii):$ immediately follows from Proposition \ref{prop3point1be1robust}\\
    $(ii) \iff (iii):$ Suppose there is a non zero scaler $\left\{a_i \right\}_{i=1}^\infty \in \ell^2$ such that $\sum\limits_{i=1}^\infty a_i P e_i = 0.$ This implies that $\sum\limits_{i=1}^\infty a_i e_i = \sum\limits_{i=1}^\infty a_i (I - P) e_i.$ Note that $\sum\limits_{i=1}^\infty a_i e_i$ convergent since $\{a_i\}_{i=1}^\infty \in \ell^2.$ Taking, $f=\{a_i\}_{i=1}^\infty,$ we have $\langle f,e_i \rangle = a_i \neq 0,\,\forall i.$ \\
  Conversely, let $f= (I-P)h,$ for some $h \in \mathcal{H}$ such that $\langle (I-P)h,e_i \rangle = \langle f,e_i \rangle \neq 0.$ let us take $a_i = \langle (I-P)h,e_i \rangle \neq 0.$ Now,
    \begin{align*}
       & \;\;\sum\limits_{i=1}^\infty a_i P e_i \\&= \sum\limits_{i=1}^\infty \langle (I-P)h,e_i \rangle Pe_i \\& = \sum\limits_{i=1}^\infty \langle h,e_i \rangle Pe_i - \sum\limits_{i=1}^\infty \langle Ph,e_i \rangle Pe_i\\& = Ph -P^2 h \\&= 0
    \end{align*}
\end{proof}

Let $F=\{f_i\}_{i \in \mathcal{I}}$ be a Besselian frame for $\mathcal{H}.$ Then, there exists a Hilbert space $\mathcal{H}' \supseteq \mathcal{H}$ with an orthonormal basis $\{e_i\}_{i \in \mathcal{I}}$ such that $Pe_i = S_{F}^{-\frac{1}{2}} f_i,\;\forall i.$ This gives $S_{F}^{\frac{1}{2}}Pe_i =  f_i,\;\forall i.$ Let $\{\phi_i\}_{i \in \mathcal{J}}$ be an orthonormal basis for $\mathcal{H}.$ Then for all $j,$
\begin{align*}
    \|f_j\|^2 &= \left\langle S_{F}^{\frac{1}{2}}Pe_j, S_{F}^{\frac{1}{2}}Pe_j \right\rangle \\&= \left\langle S_{F}Pe_j, Pe_j \right\rangle \\& = \left\langle \sum \limits_{i \in \mathcal{I}} \langle Pe_j, f_i \rangle f_i, Pe_j \right\rangle \\&=  \sum \limits_{i \in \mathcal{I}} \left| \langle Pe_j, f_i \rangle \right|^2 \\&= \sum \limits_{i \in \mathcal{I}} \left| \left\langle \sum\limits_{k \in \mathcal{J}} \langle e_j, \phi_k \rangle \phi_k, f_i \right\rangle \right|^2 \\&=  \sum \limits_{i \in \mathcal{I}} \left| \sum \limits_{k \in \mathcal{J}}\langle e_j,\phi_k \rangle \langle \phi_k, f_i \rangle  \right|^2 \\&=  \sum \limits_{i \in \mathcal{I}} \left| \left\langle e_j, \sum \limits_{k \in \mathcal{J}}\langle f_i,\phi_k \rangle \phi_k \right\rangle  \right|^2 \\&=  \sum \limits_{i \in \mathcal{I}} \left| \langle e_j, f_i \rangle  \right|^2
\end{align*}
The above expression establishes a connection between the norm of a frame element and the squared inner products of the corresponding basis elements in an extended Hilbert space. It illustrates how the norm of a frame element is determined by its expansion coefficients relative to the orthonormal basis of the extended space. This insight is particularly useful in signal processing, where frames provide redundant yet stable representations of signals. Moreover, it highlights the structural role of frames in representing vectors through their projections onto an extended orthonormal basis.

Building on this idea, the following proposition establishes a fundamental equivalence,\break demonstrating that any Besselian frame can be transformed via a unitary operator into a sequence defined by the orthonormal basis coefficients. This result underscores the structural flexibility of Besselian frames and their intrinsic connection to orthonormal bases, offering a deeper \break understanding of their geometric representation in Hilbert spaces. It reinforces the idea that frames, despite their redundancy, maintain a well-defined structure that allows for stable \break reconstruction and transformation, making them valuable in both theoretical and applied settings.

\begin{prop}
   Let $F=\{f_i\}_{i =1}^\infty$ be a Besselian frame and $\{\phi_j\}_{j =1}^\infty$  be an orthonormal basis for $\mathcal{H}.$ Then, $F$ is unitarily  equivalent to $\left\{ \sum\limits_{j=1}^\infty \langle e_i, \phi_j \rangle e_j \right\}_{i=1}^\infty.$
\end{prop}
\begin{proof}
    For a Besselian frame $F=\{f_i\}_{i =1}^\infty$ for $\mathcal{H},$ there exists a Hilbert space $\mathcal{H}' \supset \mathcal{H}$ with an Orthonormal basis $\{e_i\}_{i =1}^\infty $  such that $Pe_i = f_i,\,1\leq i \leq \infty,$ where $P$ is the orthogonal projection from $\mathcal{H}'$ onto $\mathcal{H}.$ Let $ \{\phi_j\}_{j =1}^\infty $ be an Orthonormal basis for $\mathcal{H}.$ Then, $Pe_i = \sum\limits_{j=1}^\infty \langle e_i, \phi_j \rangle \phi_j.$ Let $T$ be a unitary operator such that $T \phi_j = e_j,\;1 \leq j \leq \infty.$ Then, 
    $$Tf_i = TPe_i = \sum\limits_{j=1}^\infty \langle e_i, \phi_j \rangle T\phi_j = \sum\limits_{j=1}^\infty \langle e_i, \phi_j \rangle e_j.$$

 \end{proof}

\section{Acknowledgment}
	 The authors are grateful to the Mohapatra Family Foundation and the College of Graduate Studies of the University of Central Florida for their support during this research.

\bibliographystyle{amsplain}

\end{document}